\documentclass[12pt]{amsart}
\usepackage{graphicx}
\usepackage{epstopdf}
\usepackage{url}
\usepackage{latexsym}
\usepackage{amsfonts,amsmath,amssymb}
\newtheorem{theorem}{Theorem}[section]
\newtheorem{proposition}[theorem]{Proposition}
\newtheorem{lemma}[theorem]{Lemma}
\newtheorem{definition}[theorem]{Definition}

\newtheorem{example}[theorem]{Example}

\def\Z{\mathbb{Z}}

\date{\today}
\begin{document}

\title[Alternating runs]{Generating functions of permutations with respect to their alternating runs}

\author{Mikl\'os B\'ona}
\address{Department of Mathematics, University of Florida, $358$ Little Hall, PO Box $118105$,
Gainesville, FL, $32611-8105$ (USA)}
\email{bona@ufl.edu}

\begin{abstract} We present a short, direct proof of the fact that the generating function of all permutations of a fixed length $n\geq 4$ 
is divisible by $(1+z)^m$, where $m=\lfloor (n-2)/2 \rfloor$. \end{abstract}

\maketitle

\section{Introduction}

Let $p=p_1p_2\cdots p_n$ be a permutation. For an index $i\in [2,n-1]$,  we say that $p$ changes direction at $i$ if $p_{i-1}<p_i>p_{i+1}$ or $p_{i-1}>p_i<p_{i+1}$. Furthermore, we say that $p$ has $k$ alternating runs if $p$ changes directions a total of $k-1$ times. 

Let $\hbox{run(p)}$ be the number of alternating runs of $p$, and let 
\[R_n(z) =\sum_{p} z^{\textup{run(p)}},\]
where the sum is taken over all permutations of length $n$.
In this note, we prove that for $n\geq 4$, the polynomial  $R_n(z)$ is divisible by a high power of $(1+z)$, namely by $(1+z)^m$, where
$m=\lfloor (n-2)/2 \rfloor$. 

This result was known before, by an analytic proof given by Herbert Wilf, that can be read also in \cite{combperm}, that was based on 
the relation between the Eulerian polynomials and $R_n(z)$, and also by an induction proof by Richard Ehrenborg and the present author 
\cite{bona-ehr} that also touched upon Eulerian polynomials. However, in this paper, we provide a direct, non-inductive proof. 

\section{A group action on permutations}

Let $p=p_1p_2\cdots p_n$ be a permutation. The {\em complement} of $p$ is the permutation $\overline{p}=n+1-p_1 \ n+1 -p_2 \ \cdots n+1-p_n $. For instance, 
the complement of 425613 is 352164. It is clear that $p$ and $\overline{p}$ have the same number of alternating runs, since the diagram of $\overline{p}$ is just the diagram of $p$
reflected through a horizontal line. In what follows, we will say {\em flipped} instead of {\em reflected through a horizontal line}. 
Note that this symmetry implies that all coefficients of $R_n(z)$ are even for $n\geq 2$. 

Let $s$ be a string of entries in $p$ that are in consecutive positions. Let $S$ be set of entries that occur in $s$, in other words, 
the {\em underlying set} of $s$. Then the {\em complement of $s$ relative to  $S$} 
is the string obtained from $s$ so that for each $j$, the $j$th smallest entry of $S$ is replaced by the $j$th largest entry of $S$.   For example, if s=24783, then the complement of $s$
relative to its underlying set $S$  is 84327.

We will use a similar notion for sets. Let $T\subseteq U $. be finite sets.  Let $T=\{t_1,t_2,\cdots ,t_j\}$, where the $t_i$ are listed in increasing order. 
Let us say that $t_i$ is the $a_i$th smallest element of $U$. 
Then the {\em vertical complement} of $T$ with respect to $U$ is the set consisting of the $a_i$th largest elements of $U$, for all $i$. For instance, if $T=\{1,4,6\}$, and $U=\{1,2,3,4,6,8,9\}$, then the vertical complement of $T$ with respect to $U$ is $\{ 3,4,9\}$. Indeed, 
$T$ consists of the first, fourth, and fifth smallest elements of $U$, so its vertical complement with respect to $U$ consists of the first, fourth, 
and fifth largest elements of $U$. 

\begin{definition}
For $1\leq i\leq n$, let $c_i$ be the transformation on the set of all permutations $p=p_1p_2\cdots p_n$ that leaves the string $p_1p_2\cdots p_{i-1}$ unchanged, and replaces the string $p_ip_{i+1}\cdots  p_n$ by its complement relative to its 
underlying set.
\end{definition}

  Note that $c_1(p)=\overline{p}$ and $c_n(p)=p$ for all $p$. 

\begin{example} \label{first}  Let $p=315462$. Then $c_3(p)=314526$, while $c_5(p)=315426$.
\end{example}

\begin{proposition} \label{easysym} Let $n\geq 4$, and let $3\leq i\leq n-1$. Let $p$ be any permutation of length $n$. 
Then one of $p$ and $c_i(p)$ has exactly one more alternating run than the other. 
\end{proposition}

\begin{proof} As $c_i$ does not change the number of runs of the string $p_1p_2\cdots p_{i-1}$ or the number of runs of 
the string $p_ip_{i+1}\cdots p_n$ and its image, all changes occur within the four-element string $p_{i-2}p_{i-1}p_ip_{i+1}$. 
There are only 24 possibilities for the pattern of these four entries, and it is routine to verify the statement for each of the possible 12 pairs.
In fact, checking the six pairs in which $p_{i-2}<p_{i-1}$ is sufficient, for symmetry reasons. 
\end{proof}

In other words, $z^{\textup{run}(p)}+z^{\textup{run}(c_i(p))}$ is divisible by $1+z$.

The following lemma is crucial for our purposes. 

\begin{lemma} \label{commutative}  Let $1\leq i\leq j-2\leq n-2$. Then for all permutations $p$ of length $n$, the identity $c_i(c_j(p))=c_j(c_i(p))$ holds. 
\end{lemma}

\begin{proof}
Neither $c_i$ nor $c_j$ acts on any part of the initial segment $p_1p_2\cdots p_{i-1}$, so that segment, unchanged, will start both  $c_i(c_j(p))$ and $c_j(c_i(p))$.
The ending segment $p_jp_{j+1}\cdots p_n$ gets flipped twice by both $c_ic_j$ and $c_jc_i$, so in the end, the pattern of the last $n-j+1$ entries will be the same in  $c_i(c_j(p))$ and
 $c_j(c_i(p))$, because in both permutations, it will be the same pattern as it was in $p$. The middle segment $p_ip_{i+1}\cdots p_j$ will get flipped once by both $c_ic_j$ and $c_jc_i$,
so on both sides, the pattern of entries in positions $i,i+1,\cdots ,j$ will be the complement of the pattern of $p_ip_{i+1}\cdots p_j$.  

Finally, the {\em set} of entries in the last $n-j+1$ positions is the same in both  $c_i(c_j(p))$ and $c_j(c_i(p))$, since both sets are equal to the vertical complement of the set
 $\{p_j,p_{j+1},\cdots ,p_n\}$ with respect to the set $\{p_i,p_{j+1},\cdots ,p_n\}$ 
\end{proof}

Note that $c_i(c_j(p))\neq p$, since the segment $p_i\cdots p_{j-1}$ is of length at least two, and gets flipped at least once.

\begin{example} {\em  Continuing Example \ref{first}, we get that $c_3(c_5(315462))=c_5(c_3(315462))=314562$. } \end{example}

Now let $n\geq 4$ be any integer. If $n$ is even, 
let  \[\mathcal C_n=\{c_3, c_5,c_7,\cdots ,c_{n-1}\} .\] 
If $n$ is odd, let  \[\mathcal C_n=\{c_3, c_5,c_7,\cdots ,c_{n-2}\} .\] 
 In both cases, $\mathcal C_n$ consists of $m=\lfloor (n-2)/2 \rfloor$ operations. Each of these operations are involutions, and 
by Lemma \ref{commutative}, they pairwise commute. Therefore, they define a group $G_m \cong \Z_2^m$ that acts on the set of all permutations of length $n$. 

Note that with this set of generators,  $c_i(c_j(p))\neq p$, since the segment $p_i\cdots p_{j-1}$ is of length at least two, and gets flipped exactly once. Also note that the action of the individual $c_i$ on the set of all permutations of length $n$ is independent in the following sense. 
If $p$ is a permutation, and $c_i(p)$ has one more alternating run then $p$, then $c_j(c_i(p))$ also has one more alternating run than 
$c_j(p)$.

The action of $G_m$ on the set of all permutations of length $n$ creates orbits of size $2^m$. 

\begin{lemma} \label{orbits}
Let $A$ be any orbit of $G_m$ on the set of all permutations of length $n$. Then the equality
\[\sum_{p \in A} z^{\textup{run(p)}} = z^a(1+z)^m \]
holds, where $a$ is a nonnegative integer. 
\end{lemma}

\begin{proof} Let $p\in A$, and going through $p$ left to right, let us apply or not apply each element of $\mathcal C_n$ so as to minimize the
number of alternating runs of the obtained permutation. That is, if applying $c_i$ increases the number of alternating run, then do not apply it,
if it decreases the number of alternating runs, then apply it. Let $q$ be the obtained permutation. Then we call $q$ the {\em minimal}
permutation in $A$, since among all permutations in $A$, it is $q$ that has the smallest number of alternating runs. Now elements of $A$ 
with $i$ more alternating runs than $q$ can be obtained from $q$ by applying exactly $i$ elements of $\mathcal C_n$ to $q$. As there are
${m \choose i}$ wys to choose $i$ such elements, our statement is proved by summing over $i$. 
\end{proof} 

\begin{theorem}
For $n\geq 4$, the equality \[R_n(z) = (1+z)^m \sum_{q} z^{\textup{run(q)}} \]
holds, where the summation is over permutations $q$ that are minimal in their orbit under the action of $G_m$. Here
 $m=\lfloor (n-2)/2 \rfloor $.
\end{theorem}

\begin{proof} This follows from Lemma \ref{orbits} by summing over all orbits. \end{proof}

\end{document}